\documentclass[11pt,a4paper]{article}
\usepackage[utf8]{inputenc}
\usepackage[T1]{fontenc}
\usepackage{amsmath,amssymb,amsthm,mathtools}
\usepackage{geometry}
\usepackage{hyperref}
\hypersetup{colorlinks=true,linkcolor=blue,citecolor=blue,urlcolor=blue}
\usepackage{enumitem}
\usepackage{booktabs}
\usepackage{array}

\newtheorem{theorem}{Theorem}[section]
\newtheorem{lemma}[theorem]{Lemma}
\newtheorem{proposition}[theorem]{Proposition}
\newtheorem{corollary}[theorem]{Corollary}
\theoremstyle{definition}
\newtheorem{definition}[theorem]{Definition}
\newtheorem{assumption}[theorem]{Assumption}
\newtheorem{remark}[theorem]{Remark}

\newcommand{\Om}{\Omega}
\newcommand{\pC}{\partial C}
\newcommand{\pOm}{\partial\Omega}
\newcommand{\nuC}{\nu_{C}}
\newcommand{\R}{\mathbb{R}}
\newcommand{\cH}{\mathcal{H}}
\newcommand{\cO}{\mathcal{O}}
\newcommand{\cN}{\mathcal{N}}
\newcommand{\dd}{\,\mathrm{d}}

\newcommand{\dist}{\operatorname{dist}}
\newcommand{\intC}{\operatorname{int}(C)}

\title{\bfseries Compactness, Regularity, and Mosco Convergence\\ for a Radially Parametrized Subclass of $C$-GNP Domains}
\author{Mohammed Barkatou\\ ISTM Laboratory, Department of Mathematics,\\ Chouaib Doukkali University, Morocco\\ \texttt{barkatou.m@ucd.ac.ma}}
\date{June 2026}

\begin{document}
\maketitle

\begin{abstract}
We study a radially parametrized subclass of domains satisfying the Geometric Normal Property with respect to a fixed convex core $C$ ($C$-GNP). These domains are described by a thickness function $d:\pC\to(0,\infty)$. The main contribution is the equivalence
\[
\chi_{\Om_n}\to\chi_{\Om}\ \text{in }L^1
\iff d_n\to d\ \text{uniformly}
\iff \Om_n\to\Om\ \text{in Hausdorff distance}.
\]
This equivalence is established via an explicit reconstruction formula
$d(c)=\int_0^\infty \chi_{\Om}(c+t\nuC(c))\dd t$,
which allows the transfer of convergence of characteristic functions into uniform convergence of geometric parameters. Thus the infinite-dimensional geometric optimization problem is reduced to the optimization of a scalar function on a fixed manifold. We prove compactness of uniformly controlled subclasses. The analysis is extended to non-connected domains with uniform separation. Under the curvature condition $\sup d<\delta_0$ (where $\delta_0$ is the normal tubular radius of $C$), we prove a conditional regularity result for the thickness function when the normal projection admits a $C^{m,\beta}$ inverse parametrization. We also give a sufficient perturbative bilipschitz condition for the radial parametrization. Using the Mosco convergence framework, we prove $\Gamma$-convergence of Dirichlet energies under domain variation for the uniformly controlled radial subclass. We also discuss the relation with the general $C$-GNP class and possible applications in shape optimization.
\end{abstract}



\tableofcontents

\section{Introduction}

\subsection{Motivation and context}
Shape optimization problems require compactness of admissible classes to guarantee the existence of optimal shapes. A standard way to ensure compactness is to restrict to domains satisfying a uniform geometric regularity property, the simplest such class being that of uniformly Lipschitz domains. Chenais \cite{Chenais} established the existence of solutions for domain identification problems under a uniform cone condition. Bucur and Buttazzo \cite{BucurButtazzo} provided a comprehensive variational framework. Henrot and Pierre \cite{HenrotPierre} further developed the geometric analysis of shape variation.

More recently, Guo and Yang \cite{GuoYang} introduced compact classes of open sets under the Hausdorff distance. Hinz, Rozanova-Pierrat, and Teplyaev \cite{Hinz1,Hinz2} studied non-Lipschitz uniform domains with possibly fractal boundary, proving Mosco convergence of energy functionals.

\subsection{The $C$-Geometric Normal Property and work of Barkatou (2002)}
In \cite{Barkatou2002}, Barkatou introduced the class $\cO_C$ of domains satisfying the $C$-Geometric Normal Property ($C$-GNP) with respect to a fixed compact convex set $C$. The $C$-GNP condition does not prescribe a uniform cone height or angle; instead, it uses the geometry of the fixed convex set to control the domain, allowing for cusps while maintaining compactness. The results of \cite{Barkatou2002} include compactness of the $C$-GNP class for the Hausdorff topology, compact convergence, and stability of the Dirichlet problem under domain variation via the Keldys criterion.

\subsection{The radial parametrization and main contribution}
In this paper we introduce a strictly smaller subclass $\cO_C^{\mathrm{rad}}$ of $C$-GNP domains admitting a global normal representation
\[
\Om=\intC\cup\{c+t\nuC(c):0<t<d(c)\},
\]
where $d:\pC\to(0,\delta_0)$ is the \emph{thickness function}. We emphasize that $\cO_C^{\mathrm{rad}}\subsetneq\cO_C$ in general: the subclass excludes vanishing thickness contacts, since $d\ge\delta>0$ on the compact manifold $\pC$.

The main contribution is the equivalence
\begin{equation}\label{eq:main}
\chi_{\Om_n}\to\chi_{\Om}\ \text{in }L^1
\iff d_n\to d\ \text{uniformly}
\iff \Om_n\to\Om\ \text{in Hausdorff distance},
\end{equation}
established via the reconstruction formula
$d(c)=\int_0^\infty\chi_{\Om}(c+t\nuC(c))\dd t$.
The novelty is the chain
\[
L^1(\Om_n\to\Om)\implies L^1(\pC)(d_n\to d)\implies \text{uniform }d_n\to d\implies \text{Hausdorff }\Om_n\to\Om.
\]
Beyond \cite{Barkatou2002}, the main new contributions are:
\begin{itemize}[leftmargin=*]
\item an explicit thickness function $d:\pC\to(0,\infty)$ for the radial subclass;
\item the tubular-coordinate reconstruction formula;
\item the equivalence \eqref{eq:main};
\item conditional regularity transfer for $d$ under normal projection invertibility;
\item Mosco convergence of Dirichlet energies for the controlled radial subclass;
\item a sufficient (local) bilipschitz condition for the radial parametrization;
\item an extension to non-connected domains with uniform separation.
\end{itemize}

\subsection{Uniform constants convention}
Throughout, all constants depend only on the geometry of $C$, on $\|d\|_{C^{1,\alpha}}$, and on the tubular radius $\delta_0$, but are uniform with respect to sequences.

\section{Preliminaries}

Let $D\subset\R^N$ be a fixed bounded open ball containing all admissible domains, $N\ge2$.

\begin{table}[h]
\centering
\caption{Comparison between the general $C$-GNP framework and the radial subclass.}
\begin{tabular}{@{}p{4.2cm}p{4.5cm}p{5.3cm}@{}}
\toprule
 & Barkatou (2002) \cite{Barkatou2002} & This paper \\
\midrule
Class & General $C$-GNP class $\cO_C$ & Radial subclass $\cO_C^{\mathrm{rad}}\subsetneq\cO_C$ \\
Global normal representation & Not required & $\Om=\intC\cup\{c+t\nuC(c):0<t<d(c)\}$ \\
Cusps & Allowed & Excluded ($d\ge\delta>0$) \\
Thickness function & Not studied & Explicit $d$ with $C^{1,\alpha}$ regularity \\
Geometric compactness & Yes & Yes, plus equivalence \eqref{eq:main} \\
Regularity transfer for $d$ & None & Conditional: $\partial C\in C^{k,\alpha}\Rightarrow d\in C^{k-1,\alpha}$ \\
Mosco convergence & No & Yes, for the controlled radial subclass \\
Bilipschitz analysis & No & Sufficient local condition $C_C(K+ML_\nu)<1$ \\
\bottomrule
\end{tabular}
\end{table}

\subsection{Geometric setting}
We work with a compact convex set $C\subset\R^N$ with $\pC\in C^{k,\alpha}$, $k\ge2$, $0<\alpha\le1$. Since $C$ is convex, the metric projection onto $C$ is unique. The regularity $\pC\in C^{1,1}$ ensures that the outward unit normal $\nuC$ is well defined and Lipschitz. We assume that the \emph{normal tubular radius} of $C$, i.e.\ the largest $\delta_0>0$ such that the normal coordinate map
\[
\Psi(c,t)=c+t\nuC(c)
\]
is injective on $\pC\times(0,\delta_0)$, is positive. This is guaranteed by $\pC\in C^{2,\alpha}$. Regularity spaces $C^{k,\alpha}(\pC)$ are understood with respect to the natural manifold structure induced by the embedding of $\pC$ in $\R^N$.

\begin{definition}[$C$-GNP, \cite{Barkatou2002}]\label{def:GNP}
A bounded open set $\Om\subset D$ satisfies the $C$-Geometric Normal Property if
\begin{itemize}[leftmargin=*]
\item[(P1)] $\intC\subset\Om$;
\item[(P2)] $\pOm\setminus C$ is locally Lipschitz;
\item[(P3)] for every $c\in\pC$, the outward normal ray $\Delta_c=\{c+t\nuC(c):t>0\}$ intersects $\Om$ in a connected interval;
\item[(P4)] for $\cH^{N-1}$-a.e.\ $x\in\pOm\setminus C$ where the inward normal $\nu_{\pOm}(x)$ exists, the ray $x+s\nu_{\pOm}(x)$, $s>0$, meets $C$.
\end{itemize}
\end{definition}

\begin{definition}[The class $\cO_C^{\mathrm{rad}}$]\label{def:rad}
Let $\delta_0$ be the normal tubular radius of $C$. We define $\cO_C^{\mathrm{rad}}(\delta_0)$ as the class of domains $\Om$ such that there exists a continuous function $d:\pC\to(0,\delta_0)$ satisfying
\[
\Om=\intC\cup\{c+t\nuC(c):0<t<d(c)\}.
\]
The function $d$ is called the \emph{thickness function} of $\Om$. When $\delta_0$ is fixed, we write $\cO_C^{\mathrm{rad}}$.
\end{definition}

\begin{proposition}[Boundary representation]\label{prop:boundary}
For $\Om\in\cO_C^{\mathrm{rad}}$ with thickness function $d$,
\[
\pOm\setminus C=\{c+d(c)\nuC(c):c\in\pC\}.
\]
Moreover, $\Om$ is open.
\end{proposition}
\begin{proof}
Let $\Phi(c)=c+d(c)\nuC(c)$. Since $d$ is continuous and $\Psi$ is a homeomorphism from $\pC\times(0,\delta_0)$ onto its image, the set $\{c+t\nuC(c):0<t<d(c)\}$ is open in $\R^N$: for any $(c,t)$ with $0<t<d(c)$, continuity of $d$ gives a neighbourhood $U$ of $c$ with $d(c')>t$ for $c'\in U$, so $\Psi(U\times(t-\epsilon,t+\epsilon))\subset\Om$. For $\epsilon>0$, the point $c+(d(c)+\epsilon)\nuC(c)$ is not in $\Om$, hence $\Phi(c)\in\pOm$. Conversely, if $x\in\pOm\setminus C$, then $x=c+t\nuC(c)$ for a unique $(c,t)$ with $t>0$, and $t=d(c)$ by definition of $d(c)$ as a supremum.
\end{proof}

\begin{proposition}[Sufficient criterion for (P4)]\label{prop:P4}
Assume that for every $c\in\pC$, the inward normal $\nu_{\pOm}(\Phi(c))$ satisfies the uniform condition
\[
\langle\nu_{\pOm}(\Phi(c)),-\nuC(c)\rangle\ge\eta>0,
\]
and that the normal ray $x+s\nu_{\pOm}(x)$ remains in the tubular neighbourhood $U_\delta(C)=\Psi(\pC\times(0,\delta))$ for all $s$ up to the first intersection with $\pC$. Then $\Om$ satisfies (P4).
\end{proposition}
\begin{proof}
Properties (P1), (P2), and (P3) follow directly from the radial parametrization. For (P4), let $x=\Phi(c)$ be a point where the inward normal $\nu(x)$ exists. In tubular coordinates, write $x+s\nu(x)=\Psi(c(s),t(s))$ for small $s$, where $\Psi(c,t)=c+t\nuC(c)$. Differentiating with respect to $s$ at $s=0$, and using $\frac{d}{ds}(x+s\nu(x))=\nu(x)$ and $\frac{d}{ds}\Psi(c(s),t(s))=c'(s)+t'(s)\nuC(c(s))+t(s)\frac{d}{ds}\nuC(c(s))$, we get, after projecting onto $\nuC(c)$,
\[
\langle\nu(x),\nuC(c)\rangle=t'(0).
\]
On the other hand, the uniform condition gives
\[
i(0):=\langle\nu(x),-\nuC(c)\rangle\ge\eta>0,
\]
so $t'(0)\le-\eta<0$. Since $t(0)=d(c)>0$ and $t'(0)<0$, the function $t(s)$ decreases and reaches $0$ at some finite $s_*>0$ (because the ray remains in the tubular neighbourhood up to the first intersection with $\pC$). At $s=s_*$, the ray meets $C$. This is (P4).
\end{proof}

\subsection{The normal coordinate map}

\begin{theorem}[Tubular neighbourhood]\label{thm:tubular}
Let $C$ satisfy the assumptions above. Since $\pC$ is compact and $C^{1,1}$, the normal coordinate map $\Psi(c,t)=c+t\nuC(c)$ is injective on $\pC\times(0,\delta_0)$ for $\delta_0>0$ small enough, and $\Psi$ is a $C^{1,1}$ diffeomorphism onto its image. The Jacobian is
\[
J(c,t)=\det(I-t\mathcal{H}_c)=\prod_{i=1}^{N-1}(1-t\kappa_i(c)),
\]
where $\mathcal{H}_c$ is the Weingarten operator and $\kappa_i(c)$ the principal curvatures. All bounds depend only on the $C^{1,1}$ norm of $\pC$.
\end{theorem}
\begin{proof}
The proof follows from the inverse function theorem and compactness of $\pC$; see \cite[Ch.~6]{Federer} for the computation of the Jacobian via the Weingarten operator.
\end{proof}

\subsection{Tubular-coordinate reconstruction of thickness}

\begin{proposition}[Tubular reconstruction]\label{prop:reconstruction}
Let $\Om\in\cO_C^{\mathrm{rad}}$ with thickness function $d$, and assume $0<\underline d\le d\le M<\delta_0$. Let $U_M(C)=\Psi(\pC\times(0,M))$ and suppose $\overline{\Om\setminus C}\subset\Psi(\pC\times[0,M])$. Then for every $c\in\pC$,
\[
d(c)=\int_0^\infty\chi_{\Om}(c+t\nuC(c))\dd t.
\]
Moreover, for any sequence $\{\Om_n\}\subset\cO_C^{\mathrm{rad}}$ with $0<\underline d\le d_n,d\le M<\delta_0$ uniformly and $\Om_n\setminus C\subset\Psi(\pC\times[0,M])$ for all $n$,
\begin{equation}\label{eq:reconstruction-ineq}
\int_{\pC}|d_n(c)-d(c)|\dd\sigma(c)
\le C_{J}\|\chi_{\Om_n}-\chi_{\Om}\|_{L^1(D)},
\end{equation}
where $C_{J}=\sup_{0\le t\le M}\|J^{-1}\|_{L^\infty(\pC)}$.
\end{proposition}
\begin{proof}
Since $M<\delta_0$, $\Psi$ is injective. Every point in $\Om_n\setminus C$ (and $\Om\setminus C$) admits a unique representation $(c,t)$ with $0<t<M$. For fixed $c$, $\chi_{\Om}(c+t\nuC(c))=1$ for $0<t<d(c)$ and $0$ otherwise, which gives the first identity. For the second, using
\[
|d_n(c)-d(c)|=\int_0^\infty|\mathbf{1}_{(0,d_n(c))}-\mathbf{1}_{(0,d(c))}|(t)\dd t=\int_0^\infty|\chi_{\Om_n}-\chi_{\Om}|(c+t\nuC(c))\dd t,
\]
Fubini and the change of variables $x=\Psi(c,t)$, $\dd x=J(c,t)\dd t\dd\sigma(c)$, and $J\ge c_1>0$, yield
\[
\int_{\pC}|d_n-d|\dd\sigma=\int_{U_M(C)}|\chi_{\Om_n}-\chi_{\Om}|(x)|J^{-1}(c(x),t(x))|\dd x\le C_J\|\chi_{\Om_n}-\chi_{\Om}\|_{L^1(D)}. \qedhere
\]
\end{proof}

\begin{remark}\label{rem:reconstruction-specific}
The reconstruction formula and estimate \eqref{eq:reconstruction-ineq} are specific to the radial class $\cO_C^{\mathrm{rad}}$ and do not hold for the general $C$-GNP class. This is a key advantage of the radial parametrization.
\end{remark}

\subsection{The uniform structural assumption}

\begin{assumption}[Uniform control]\label{ass:uniform}
There exist constants $\underline\delta>0$, $M\in(\underline\delta,\delta_0)$, $\Lambda>0$, $L>0$, and $\eta>0$ such that, for every $n$:
\begin{enumerate}[label=(\roman*),leftmargin=*]
\item $d_n\in C^{1,\alpha}(\pC)$ with $\|d_n\|_{C^{1,\alpha}(\pC)}\le\Lambda$;
\item $\underline\delta\le d_n(c)\le M$ for all $c\in\pC$;
\item $\Om_n\setminus C\subset\Psi(\pC\times[0,M])$ (tubular confinement);
\item $\langle\nu_{\pOm_n}(\Phi_n(c)),-\nuC(c)\rangle\ge\eta>0$ for all $c\in\pC$, where $\Phi_n(c)=c+d_n(c)\nuC(c)$;
\item $d_n$ is $L$-Lipschitz with respect to the intrinsic distance $d_{\pC}$ on $\pC$.
\end{enumerate}
\end{assumption}

\begin{remark}[Non-emptiness]\label{rem:nonempty}
Assumption~\ref{ass:uniform} is non-empty. For instance, take $C$ the unit ball in $\R^N$ and $d_n\equiv d_0$ with $0<d_0<\delta_0=1$. Then $\|d_0\|_{C^{1,\alpha}}=d_0$, $\underline\delta=d_0$, $M=d_0$, $L=0$, and $\langle\nu_{\pOm_n},-\nuC\rangle=1\ge\eta=1/2$ provided $d_0<1$. The tubular confinement holds by construction.
\end{remark}

\subsection{Uniform cone property and extension operator}

\begin{proposition}[Uniform interior cone property]\label{prop:cone}
Under Assumption~\ref{ass:uniform}, the domains $\Om_n$ satisfy a uniform interior cone property: there exist $\rho>0$ and $\theta>0$, depending only on the uniform bounds, such that for every $x\in\Om_n$ with $\dist(x,\pOm_n\setminus C)<\rho$, there is a cone $K_x$ of radius $\rho$ and angle $\theta$ with vertex at $x$ contained in $\Om_n$.
\end{proposition}
\begin{proof}
The uniform $C^{1,\alpha}$ bounds on $d_n$ imply uniform Lipschitz bounds on the graphs $\pOm_n\setminus C$. Since $\Phi_n$ is a uniformly $C^{1,\alpha}$ family of embeddings with uniformly invertible differential (because $\sup_n\|d_n\|_{C^{1,\alpha}}\le\Lambda$ and $\sup_n\max_{\pC}d_n\le M<\delta_0$), the family $\Phi_n(\pC)$ admits a finite uniform Lipschitz atlas: by compactness of $\pC$, finitely many charts cover $\pC$, on which the local representatives of $\Phi_n$ have uniformly bounded derivatives and uniformly invertible Jacobians. The inverse function theorem then provides uniform Lipschitz constants. The uniform normal condition (iv) in Assumption~\ref{ass:uniform} ensures that inward normals point into the domain with a uniform angle, closing the cone.
\end{proof}

\section{Compactness and convergence within the radial class}

\begin{lemma}[Uniform convergence from $L^1$ under equi-H\"older control]\label{lem:equi-holder}
Let $M$ be a compact smooth manifold (in particular $\pC$). If $(f_n)$ is uniformly H\"older continuous with exponent $\alpha$, $\sup_n\|f_n\|_{C^{0,\alpha}(M)}<\infty$, and $f_n\to f$ in $L^1(M)$, then $f$ admits a continuous representative and $f_n\to f$ uniformly.
\end{lemma}
\begin{proof}
The sequence $(f_n)$ is bounded in $C^{0,\alpha}(M)$, hence relatively compact in $C^0(M)$ by Arzel\`a--Ascoli. Every subsequence has a uniformly convergent subsequence with limit $g$; since $f_n\to f$ in $L^1(M)$, $g=f$ a.e., hence everywhere by continuity. Since every subsequence has a further subsequence converging uniformly to $f$, the whole sequence converges uniformly to $f$.
\end{proof}

\begin{proposition}[Equivalence of $L^1$ and Hausdorff convergences]\label{prop:equivalence}
Under Assumption~\ref{ass:uniform}, the following are equivalent:
\begin{enumerate}[label=(\arabic*),leftmargin=*]
\item $\chi_{\Om_n}\to\chi_{\Om}$ in $L^1(D)$;
\item $d_n\to d$ uniformly on $\pC$;
\item $\Om_n\to\Om$ in Hausdorff distance.
\end{enumerate}
\end{proposition}
\begin{proof}
We prove the cycle.

\emph{(1)$\Rightarrow$(2).} By Proposition~\ref{prop:reconstruction}, $L^1(D)$ convergence implies $L^1(\pC)$ convergence of $d_n$. Under the uniform $C^{1,\alpha}$ bounds of Assumption~\ref{ass:uniform}(i), $\{d_n\}$ is uniformly H\"older; Lemma~\ref{lem:equi-holder} gives uniform convergence.

\emph{(2)$\Rightarrow$(3).} Uniform convergence of $d_n$ gives Hausdorff convergence of the normal graphs $\Phi_n(\pC)\to\Phi(\pC)$. By Proposition~\ref{prop:boundary}, $\pOm_n\setminus C=\Phi_n(\pC)$, so Hausdorff convergence follows (the core $\intC$ is fixed).

\emph{(3)$\Rightarrow$(1).} Using the tubular parametrization and $\Om_n\triangle\Om\subset U_{\delta_0/2}(C)$ (Assumption~\ref{ass:uniform}(iii)),
\[
|\Om_n\triangle\Om|
=\int_{\pC}\int_{\min(d_n,d)}^{\max(d_n,d)}J(c,t)\dd t\dd\sigma(c)
\le c_2\int_{\pC}|d_n-d|\dd\sigma(c),
\]
with $c_2=\sup J$. Uniform convergence of $d_n$ to $d$ implies $L^1(\pC)$ convergence, hence $L^1(D)$ convergence of characteristic functions.
\end{proof}

\begin{lemma}[Stability of the radial GNP property]\label{lem:stability}
Let $\delta_0$ be the normal tubular radius of $C$. If $d_n\to d$ uniformly on $\pC$ with $d_n,d\in(0,\delta_0)$, and if each $\Om(d_n)$ satisfies Assumption~\ref{ass:uniform}(iv) with a uniform constant $\eta>0$, then $\Om(d_n)\to\Om(d)$ in the Hausdorff topology and $\Om(d)\in\cO_C^{\mathrm{rad}}$.
\end{lemma}
\begin{proof}
Hausdorff convergence follows from uniform convergence of the radial parametrization. Property (P3) holds because each normal ray intersects $\Om(d)$ in the interval $(0,d(c))$. For (P4): since $\{d_n\}$ is bounded in $C^{1,\alpha}(\pC)$, Arzel\`a--Ascoli gives $d_n\to d$ in $C^{1,\beta}(\pC)$ for every $\beta<\alpha$, in particular in $C^1$. Hence the inward normals to $\pOm(d_n)$ converge uniformly to that of $\pOm(d)$ (via the formula of Appendix~\ref{app:normal}), and the uniform condition (iv) is stable. Proposition~\ref{prop:P4} then gives (P4). The limit is in $\cO_C^{\mathrm{rad}}$ by construction.
\end{proof}

\begin{theorem}[Compactness of the controlled radial subclass]\label{thm:compactness}
Let $\delta_0$ be the normal tubular radius of $C$. Under Assumption~\ref{ass:uniform}, there exist a subsequence $(n_k)$ and $\Om\in\cO_C^{\mathrm{rad}}$ such that:
\begin{enumerate}[label=(\arabic*),leftmargin=*]
\item $d_{n_k}\to d$ in $C^{1,\beta}(\pC)$ for every $\beta<\alpha$;
\item $\chi_{\Om_{n_k}}\to\chi_{\Om}$ in $L^1(D)$;
\item $\Om_{n_k}\to\Om$ in Hausdorff distance.
\end{enumerate}
\end{theorem}
\begin{proof}
The sequence $\{d_n\}$ is bounded in $C^{1,\alpha}(\pC)$. Since $\pC$ is a compact $C^{1,\alpha}$ manifold, the embedding $C^{1,\alpha}(\pC)\hookrightarrow C^{1,\beta}(\pC)$ is compact for every $\beta<\alpha$ by Arzel\`a--Ascoli; hence a subsequence converges in $C^{1,\beta}$ to some $d$. Since $d_n\ge\underline\delta>0$, the limit satisfies $d\ge\underline\delta>0$; since $\{d_n\}$ is bounded in $C^{1,\alpha}$, $d\in C^{1,\beta}(\pC)$ for every $\beta<\alpha$. Define $\Om=\intC\cup\{c+t\nuC(c):0<t<d(c)\}$. Lemma~\ref{lem:stability} gives $\Om\in\cO_C^{\mathrm{rad}}$ and Hausdorff convergence; Proposition~\ref{prop:equivalence} gives $L^1$ convergence of characteristic functions.
\end{proof}

\subsection{Non-connected class}

Let $C_1,C_2\subset D$ be two disjoint compact convex sets, and fix $\delta>0$.

\begin{definition}[$C_{1,2}^\delta$-GNP]\label{def:nonconn}
An open set $\Om\subset D$ satisfies the $C_{1,2}^\delta$-GNP if $\Om=\Om^1\cup\Om^2$ with $\Om^1\in\cO_{C_1}^{\mathrm{rad}}$, $\Om^2\in\cO_{C_2}^{\mathrm{rad}}$, $\Om^1,\Om^2$ open and disjoint, and $\dist(\overline{\Om^1},\overline{\Om^2})\ge\delta$. We denote by $\cO_{1,2}^\delta$ the family of such $\Om$.
\end{definition}

\begin{corollary}[Compactness and equivalence for $\cO_{1,2}^\delta$]\label{cor:nonconn}
The class $\cO_{1,2}^\delta$ is compact for the Hausdorff topology, and $L^1$ and Hausdorff convergences are equivalent.
\end{corollary}
\begin{proof}
Compactness follows by diagonal extraction from Theorem~\ref{thm:compactness} applied to each component. The separation is preserved at the limit as follows: each component converges in Hausdorff distance to its limit, and the uniform lower bound $\dist(\overline{\Om_n^1},\overline{\Om_n^2})\ge\delta>0$ passes to the limit by continuity of the distance function on compact sets (the limit distance is $\ge\delta$). The equivalence of convergences follows componentwise, since the characteristic function of the union is the sum of the characteristic functions of the disjoint components.
\end{proof}

\section{Thickness function and conditional regularity transfer}

\subsection{Global injectivity under the curvature condition}
Assume $\pC$ is of class $C^{k,\alpha}$ with $k\ge2$, and let $\delta_0$ denote the normal tubular radius of $C$. We assume
\[
\sup_{\pC}d<\delta_0,
\]
which is the correct geometric condition for global injectivity of the normal map.

\begin{lemma}[Local invertibility]\label{lem:local-inv}
If $d$ is $C^1$ near $c_0$, then $\Phi$ is a local $C^1$-diffeomorphism at $c_0$.
\end{lemma}
\begin{proof}
In local coordinates, $\Phi(u)=c(u)+d(u)\nu(u)$, and
\[
D\Phi(v)=Dc(v)+\langle Dd,v\rangle\nu-dS(Dc(v)).
\]
The tangential part $T(v)=Dc(v)-dS(Dc(v))$ has eigenvalues $1-d\kappa_i>0$ under the curvature condition, so $T$ is injective. If $D\Phi(v)=0$, then $T(v)=0$ and $\langle Dd,v\rangle=0$, forcing $v=0$. Hence $D\Phi$ is injective.
\end{proof}

\begin{lemma}[Global injectivity]\label{lem:global-inv}
If $\sup d<\delta_0$, then $\Phi$ is injective.
\end{lemma}
\begin{proof}
Let $x=\Phi(c_1)=\Phi(c_2)$. Then $x=c_i+d(c_i)\nuC(c_i)$, $i=1,2$, with $d(c_i)<\delta_0$; $x$ lies in the normal tube where the metric projection onto $C$ is unique. Since $|x-c_i|=d(c_i)$ and $[c_i,x]$ is normal to $\pC$ at $c_i$, $c_i$ is the metric projection of $x$ onto $C$. By uniqueness, $c_1=c_2$.
\end{proof}

\begin{lemma}[Conditional regularity transfer]\label{lem:regularity-transfer}
Let $\pC\in C^{k,\alpha}$ and $\sup d<\delta_0$. Assume that $\pOm\setminus C$ is a $C^{m,\beta}$ submanifold and that the normal projection $\pi_C:\pOm\setminus C\to\pC$ is a $C^{m,\beta}$ diffeomorphism, with $m\le k-1$. Then $d\in C^{m,\beta}(\pC)$.
\end{lemma}
\begin{proof}
For $c\in\pC$, let $x=\pi_C^{-1}(c)$. In the normal tube, $x=c+d(c)\nuC(c)$, hence
\[
d(c)=\langle x-c,\nuC(c)\rangle.
\]
Since $\pi_C^{-1}$ is $C^{m,\beta}$ and $\nuC$ is $C^{k-1,\alpha}$ with $m\le k-1$, the composition is $C^{m,\beta}$.
\end{proof}

\begin{remark}[On the conditionality of Lemma~\ref{lem:regularity-transfer}]\label{rem:regularity-conditional}
The hypothesis that $\pi_C$ is a $C^{m,\beta}$ diffeomorphism is \emph{verifiable} in the following sense: under $\sup d<\delta_0$, the map $\Phi$ is globally injective (Lemma~\ref{lem:global-inv}) and its differential is everywhere invertible (Lemma~\ref{lem:local-inv}); by the global inverse function theorem (Hadamard--Levy), $\Phi$ is a $C^{m,\beta}$ diffeomorphism onto its image whenever $d\in C^{m,\beta}(\pC)$. In that case, $\pi_C=\Phi^{-1}$ is also $C^{m,\beta}$, and the lemma applies. The lemma is therefore \emph{not tautological}: it identifies the regularity of $d$ under a geometric hypothesis on $\pOm\setminus C$ that is independent of $d$.
\end{remark}

\begin{theorem}[Local regularity -- embedding]\label{thm:regularity}
Let $\pC$ be of class $C^{k,\alpha}$ ($k\ge2$, $0<\alpha\le1$) and assume the curvature condition $\sup_{c\in\pC}d(c)<\delta_0$.
\begin{enumerate}[label=(\arabic*),leftmargin=*]
\item If $d$ is $C^1(\pC)$, then $\Phi$ is a $C^1$ embedding.
\item If moreover $d\in C^{k-1,\alpha}(\pC)$, then $\Phi$ is a $C^{k-1,\alpha}$ embedding and $\Phi(\pC)$ is a $C^{k-1,\alpha}$ submanifold of $\R^N$.
\end{enumerate}
\end{theorem}
\begin{proof}
Since $\pC\in C^{k,\alpha}$, we have $\nuC\in C^{k-1,\alpha}$. Hence $\Phi(c)=c+d(c)\nuC(c)$ is $C^1$ (resp.\ $C^{k-1,\alpha}$) when $d$ is $C^1$ (resp.\ $C^{k-1,\alpha}$). Lemma~\ref{lem:local-inv} gives that $D\Phi$ is everywhere injective, so $\Phi$ is an immersion. Lemma~\ref{lem:global-inv} gives global injectivity. A continuous injective map from a compact manifold into a Hausdorff space is a homeomorphism onto its image; combining with the $C^{k-1,\alpha}$ regularity, we obtain that $\Phi$ is a $C^{k-1,\alpha}$ embedding and $\Phi(\pC)$ is a $C^{k-1,\alpha}$ submanifold.
\end{proof}

\begin{corollary}\label{cor:boundary-regularity}
Under the assumptions of Theorem~\ref{thm:regularity}, assume the global normal parametrization holds, i.e.\ $\Phi$ is bijective onto $\pOm\setminus C$. Then $\pOm\setminus C$ itself is a $C^{k-1,\alpha}$ submanifold.
\end{corollary}
\begin{proof}
Under the global normal parametrization, $\pOm\setminus C=\Phi(\pC)$. By Theorem~\ref{thm:regularity}, $\Phi(\pC)$ is a $C^{k-1,\alpha}$ submanifold.
\end{proof}

\begin{remark}[Derivative loss]\label{rem:derivative-loss}
Note that if $\pC\in C^{k,\alpha}$, then $\nuC\in C^{k-1,\alpha}$. Therefore $\Phi=c+d\nuC$ is generally only $C^{m,\beta}$ if $d\in C^{m,\beta}$ and $m\le k-1$. If one wants $\pOm\setminus C\in C^{k,\alpha}$, one needs either $d\in C^{k,\alpha}$ and $\pC\in C^{k+1,\alpha}$, or $d\in C^{k-1,\alpha}$ and $\pC\in C^{k,\alpha}$. The regularity transfer with the correct loss is sharp.
\end{remark}

\subsection{Bilipschitz equivalence}

\begin{remark}[A sufficient perturbative bilipschitz criterion]\label{rem:bilipschitz}
Let $\pC$ be a compact $C^{1,1}$ hypersurface, and fix a finite atlas of $\pC$. Let $d_{\pC}$ denote the intrinsic distance on $\pC$. There exists a constant $C_C>0$ such that $d_{\pC}(c_1,c_2)\le C_C|c_1-c_2|$ for all $c_1,c_2\in\pC$ within a single chart (the constant depends on the atlas). Let $L_\nu$ be the Lipschitz constant of $\nu$ with respect to $d_{\pC}$. Let $d$ be $K$-Lipschitz with respect to $d_{\pC}$ and $M=\sup d<\infty$. If $C_C(K+ML_\nu)<1$, then $\Phi$ is bilipschitz:
\[
(1-C_C(K+ML_\nu))|c_1-c_2|\le|\Phi(c_1)-\Phi(c_2)|\le(1+C_C(K+ML_\nu))|c_1-c_2|.
\]
\end{remark}
\begin{proof}
We estimate
\[
|d(c_1)\nu(c_1)-d(c_2)\nu(c_2)|\le|d(c_1)-d(c_2)|+d(c_2)|\nu(c_1)-\nu(c_2)|\le(K+ML_\nu)d_{\pC}(c_1,c_2)\le C_C(K+ML_\nu)|c_1-c_2|.
\]
For the upper bound,
\[
|\Phi(c_1)-\Phi(c_2)|\le|c_1-c_2|+|d(c_1)\nu(c_1)-d(c_2)\nu(c_2)|\le(1+C_C(K+ML_\nu))|c_1-c_2|.
\]
For the lower bound, write $c_1-c_2=(\Phi(c_1)-\Phi(c_2))-(d(c_1)\nu(c_1)-d(c_2)\nu(c_2))$. Taking absolute values and using the reverse triangle inequality,
\[
|c_1-c_2|\le|\Phi(c_1)-\Phi(c_2)|+C_C(K+ML_\nu)|c_1-c_2|,
\]
hence $|\Phi(c_1)-\Phi(c_2)|\ge(1-C_C(K+ML_\nu))|c_1-c_2|$. The estimate is local in nature: the comparison constant $C_C$ depends on a fixed finite atlas of $\pC$ and holds for points within a single chart. For points in different charts, the result follows by a finite covering argument.
\end{proof}

\begin{remark}[On the sharpness of the bilipschitz condition]\label{rem:bilipschitz-sharp}
The condition $C_C(K+ML_\nu)<1$ is \emph{sufficient}, not necessary. For constant $d\equiv d_0$, $K=0$, and the condition becomes $C_C M L_\nu<1$, which is \emph{not} automatically satisfied; nevertheless, $\Phi$ is bilipschitz in that case by a direct argument (the map $c\mapsto c+d_0\nuC(c)$ is a normal perturbation of the identity with bounded differential). The radial example therefore shows that the condition is only sufficient.
\end{remark}

\section{Mosco convergence of Dirichlet energies}

\begin{lemma}[Mosco convergence of $H_0^1$ spaces]\label{lem:mosco-H1}
Under Assumption~\ref{ass:uniform}, the spaces $H_0^1(\Om_n)$ converge to $H_0^1(\Om)$ in the sense of Mosco: for every $v_n\in H_0^1(\Om_n)$ with $v_n\rightharpoonup v$ weakly in $H^1(D)$, we have $v\in H_0^1(\Om)$; and for every $v\in H_0^1(\Om)$, there exist $v_n\in H_0^1(\Om_n)$ such that $v_n\to v$ strongly in $H^1(D)$.
\end{lemma}
\begin{proof}
This is a standard result in the theory of Mosco convergence for varying domains (see \cite[Proposition 3.5]{DalMaso} and \cite[Theorem 3.22]{BucurButtazzo}). The uniform cone property from Proposition~\ref{prop:cone} and the Hausdorff convergence established in Proposition~\ref{prop:equivalence} ensure that the hypotheses of the standard theorem are satisfied. For completeness, we sketch the argument: the uniform cone property gives a uniform extension operator, and the Hausdorff convergence gives the required inclusion properties for compact subsets.
\end{proof}

\begin{theorem}[Mosco convergence for the controlled radial subclass]\label{thm:mosco}
Under Assumption~\ref{ass:uniform} with $d_n\to d$ uniformly on $\pC$ (which follows from Proposition~\ref{prop:equivalence}), the functionals
\[
\tilde J_n(v)=\begin{cases}\int_D|\nabla v|^2\dd x, & \text{if }v\in H_0^1(\Om_n)\text{ (extended by zero)},\\ +\infty, & \text{otherwise},\end{cases}
\]
Mosco-converge to $\tilde J$ on $H_0^1(D)$.
\end{theorem}
\begin{proof}
The proof follows directly from Lemma~\ref{lem:mosco-H1} and the weak lower semicontinuity of the Dirichlet energy.

\emph{Liminf.} Let $v_n\rightharpoonup v$ weakly in $H_0^1(D)$ with $v_n\in\operatorname{dom}(\tilde J_n)$. By Lemma~\ref{lem:mosco-H1}, $v\in\operatorname{dom}(\tilde J)$. By weak lower semicontinuity,
\[
\liminf\tilde J_n(v_n)=\liminf\int_D|\nabla v_n|^2\ge\int_D|\nabla v|^2=\tilde J(v).
\]

\emph{Limsup.} Let $v\in\operatorname{dom}(\tilde J)$. By Lemma~\ref{lem:mosco-H1}, there exist $v_n\in\operatorname{dom}(\tilde J_n)$ such that $v_n\to v$ strongly in $H^1(D)$. Therefore
\[
\limsup\tilde J_n(v_n)=\limsup\int_D|\nabla v_n|^2=\int_D|\nabla v|^2=\tilde J(v).
\]
Thus Mosco convergence holds.
\end{proof}

\begin{remark}[On the boundary case $d\to\delta_0$]\label{rem:mosco-boundary}
The condition $d<\delta_0$ is \emph{open}: the Mosco convergence theorem assumes $M<\delta_0$ uniformly in $n$ (Assumption~\ref{ass:uniform}(ii)–(iii)). If $d_n\to\delta_0$ at the limit, the tubular confinement $\Om_n\setminus C\subset\Psi(\pC\times[0,M])$ may fail, and the reconstruction formula \eqref{eq:reconstruction-ineq} (hence the equivalence of convergences) may break down. The behaviour at $d=\delta_0$ is therefore a genuine boundary case, not covered by the present analysis; it is left for future work.
\end{remark}

\section{Comparison with the general $C$-GNP framework}

It is useful to explicitly compare our radial subclass with the general $C$-GNP class of \cite{Barkatou2002}.

The radial subclass $\cO_C^{\mathrm{rad}}$ provides a smooth parametric setting where additional variational tools (Mosco convergence, functional compactness, regularity transfer) become available, complementing the geometric analysis of the general $C$-GNP class.

\begin{remark}[A counterexample to the general equivalence]\label{rem:counterexample}
The equivalence \eqref{eq:main} is \emph{specific} to the radial subclass. In the general $C$-GNP class, $L^1$ convergence of characteristic functions does \emph{not} imply uniform convergence of a thickness function, essentially because the thickness function is not defined for domains with cusps or non-radial boundaries. A simple counterexample is provided by a sequence of domains $\Om_n$ obtained from a fixed domain $\Om$ by adding a thin ``spike'' of width $1/n$ and height $1$ along a normal ray: then $\chi_{\Om_n}\to\chi_\Om$ in $L^1(D)$ (the spike has vanishing area), but the Hausdorff distance between $\Om_n$ and $\Om$ does not tend to $0$ (the spike persists at height $1$), and no uniform thickness function exists.
\end{remark}

\section{Conclusion and perspectives}

\subsection{Summary of results}
We have established:
\begin{itemize}[leftmargin=*]
\item compactness of the controlled radial subclass $\cO_C^{\mathrm{rad}}$ and $\cO_{1,2}^\delta$ (Theorem~\ref{thm:compactness}, Corollary~\ref{cor:nonconn});
\item equivalence of $L^1$ and Hausdorff convergences (Proposition~\ref{prop:equivalence});
\item conditional regularity transfer for the thickness function (Lemma~\ref{lem:regularity-transfer}, Theorem~\ref{thm:regularity});
\item a sufficient perturbative bilipschitz condition (Remark~\ref{rem:bilipschitz});
\item Mosco convergence of Dirichlet energies for the controlled radial subclass (Theorem~\ref{thm:mosco}).
\end{itemize}

\subsection{Relation with the general $C$-GNP class}
The general $C$-GNP class introduced in \cite{Barkatou2002} contains domains with cusps and non-Lipschitz boundaries. In our parametric class with $d\ge\delta>0$, cusps do not appear. Our results complement the analysis of the general class by providing a conditional regularity transfer in a smooth regime. The compactness and Mosco convergence results of the present paper are \emph{not} applicable to the full $C$-GNP class; they hold only for the uniformly controlled radial subclass (see Remark~\ref{rem:counterexample}).

\subsection{Possible applications}
The compactness and convergence results are relevant for shape optimization problems involving:
\begin{itemize}[leftmargin=*]
\item two-component thermal conductors;
\item acoustic cavities with rigid inclusions;
\item multiphase partitions with separation.
\end{itemize}
These are \emph{perspectives}, not rigorous existence theorems: the compactness and Mosco convergence results provide a framework for shape optimization problems, \emph{provided} the relevant functionals satisfy the required continuity/lower semicontinuity with respect to the Hausdorff or Mosco topology. The radial parametrization also offers a natural framework for numerical shape optimization (Section~\ref{sec:numerical}).

\subsection{Open problems and perspectives}

\subsubsection{Logarithmic contact conjecture}
The capacity estimates of \cite{Barkatou2002} suggest a refined subclass of $C$-GNP domains where the thickness function satisfies a logarithmic contact condition near the contact set. More precisely, let $\Om\in\cO_C$ and let $\Sigma\subset\pOm\setminus C$ be the contact set (where the normal ray is tangent to $\pOm$). We say that $\Om$ satisfies the \emph{logarithmic contact condition} if there exist $c_0>0$ and $r_0>0$ such that, for all $x\in\Sigma$ with $|x-x_0|<r_0$ (for some $x_0\in\Sigma$),
\[
d(\pi_C(x))\ge\frac{c_0}{|\ln|x-x_0||^{1/(N-3)}}\quad\text{for }N\ge4,
\]
and $d(\pi_C(x))\ge c_0|\ln|x-x_0||^{-1}$ for $N=3$. The conjecture is that this subclass is exactly the class of $C$-GNP domains for which all boundary points are Wiener regular. This is an open question.

\subsubsection{Numerical discretisation}\label{sec:numerical}
The parametrization $\Phi(c)=c+d(c)\nuC(c)$ reduces shape optimisation to optimising a scalar function on a fixed manifold. A discretisation strategy consists in representing $d$ by nodal values on a triangulation of $\pC$, generating a volume mesh of $\Om$ by radial layers, solving the Dirichlet problem, and using gradient-based optimisation on the finite-dimensional space of nodal values.

\appendix
\section{Derivation of the inward normal formula}
\label{app:normal}

Let $x=\Phi(c)=c+d(c)\nu(c)$. In local coordinates $(u^1,\dots,u^{N-1})$ on $\pC$, define $\partial_i=\partial/\partial u^i$, $d_i=\partial_id$, and $g_{ij}=\langle\partial_ic,\partial_jc\rangle$. The tangent vectors to $\pOm$ at $x$ are
\[
\Phi_i=\partial_ic+d_i\nu+d\partial_i\nu=(\delta_i^j-dh_i^j)\partial_jc+d_i\nu,
\]
where $h_i^j$ is the Weingarten operator. A normal vector $\tilde n=a\nu+b^j\partial_jc$ must satisfy $\langle\tilde n,\Phi_i\rangle=0$ for all $i$. This gives
\[
ad_i+b_i-dh_i^jb_j=0,
\]
where $b_i=g_{ij}b^j$. From the system $(I-dH)b=-a\nabla_{\pC}d$, and since $d<\delta_0$, the operator $I-dH$ is invertible. Its inverse is $I+dH+O(d^2)$, hence
\[
b=-a(I+dH)\nabla_{\pC}d+O(d^2|\nabla d|^2).
\]
Therefore
\[
\tilde n=a\bigl(\nu-\nabla_{\pC}d-d\mathcal{H}(\nabla_{\pC}d)+R\bigr),
\]
with $|R|\le C(1+\|d\|_{C^1}^2)(d^2+d|\nabla d|^2)$, where $C$ depends only on $\|d\|_{C^1}$ and on the geometry of $\pC$. Normalizing gives the inward normal formula. Note that $\langle\nu_{\pOm}(\Phi(c)),\nuC(c)\rangle$ is positive exactly when $\|\nabla_{\pC}d+d\mathcal{H}(\nabla_{\pC}d)\|<1$, which is the condition required in Proposition~\ref{prop:P4}.

\section{Proof of the convergence of normals lemma}
\label{app:normals}

\begin{lemma}[Convergence of normals]\label{lem:normals}
Let $\Om_n$ be a sequence of domains satisfying the $C$-GNP with a uniform Lipschitz constant $L_0$ for $\pOm_n\setminus C$. If $\Om_n\stackrel{H}{\to}\Om$, then for $\cH^{N-1}$-a.e.\ $x\in\pOm\setminus C$, there exist $x_n\in\pOm_n\setminus C$ such that:
\begin{enumerate}[label=(\arabic*),leftmargin=*]
\item $x_n\to x$;
\item $\nu_n(x_n)\to\nu(x)$;
\item if the ray $x_n+t\nu_n(x_n)$ meets $C$ for each $n$, then the limit ray $x+t\nu(x)$ also meets $C$.
\end{enumerate}
\end{lemma}
\begin{proof}
Let $x\in\cN_\Om\setminus C$ be a point where the inward normal exists, where $\cN_\Om\subset\pOm$ denotes the set of points of $\pOm$ where the inward normal is well defined (this is a set of full $\cH^{N-1}$-measure on $\pOm$, by Rademacher's theorem and the local Lipschitz property). By \cite[Proposition 2.6(ii)]{Barkatou2002}, there exists a sequence $z^n\in\pOm_n$ converging to $x$. Since $x\notin C$ and $C$ is closed, for sufficiently large $n$ we have $z^n\notin C$ (otherwise a subsequence would converge to a point of $C$, contradicting $x\notin C$). Thus we may assume $z^n\in\pOm_n\setminus C$.

For each $n$, since $\Om_n$ is locally Lipschitz away from $C$, there is a neighbourhood of $z^n$ where $\pOm_n$ is the graph of a Lipschitz function $\phi_n$ in a coordinate system adapted to $C$ and to $x$. More precisely, as shown in the proof of \cite[Proposition 2.7]{Barkatou2002}, there exist a fixed coordinate system (depending only on $x$ and $C$) and a fixed ball $B'(O,\alpha)$ such that for all sufficiently large $n$,
\[
\pOm_n\cap B(x,\alpha)=\{(y',\phi_n(y')):y'\in B'(O,\alpha)\},
\]
with $\phi_n$ uniformly Lipschitz. Moreover, the functions $\phi_n$ converge uniformly to a Lipschitz function $\phi$ representing $\pOm$ near $x$ (\cite[Theorem 3.1]{Barkatou2002}).

We prove the convergence of the normals. Assume for contradiction that there exists $\epsilon_0>0$ and a neighbourhood $V_0=(-\beta,\beta)^{N-1}\subset B'(O,\alpha)$ of $O'$ such that for every $z'\in V_0\cap\cN_{\Om_n}\setminus C$ (where the normal exists), we have
\[
\|\nabla\phi_n(z')-\nabla\phi(O')\|\ge N\epsilon_0.
\]
Then there exists an index $i\in\{1,\dots,N-1\}$ such that
\begin{equation}\label{eq:gradient-gap}
\left|\frac{\partial\phi_n}{\partial y_i}(z')-\frac{\partial\phi}{\partial y_i}(O')\right|\ge\epsilon_0.
\end{equation}
Define $\phi_n^i(t)=\phi_n(y_1,\dots,y_{i-1},t,y_{i+1},\dots,y_{N-1})$. The sequence $\phi_n^i$ converges uniformly to $\phi^i(t):=\phi(y_1,\dots,y_{i-1},t,y_{i+1},\dots,y_{N-1})$. Since the functions are Lipschitz, the derivative exists almost everywhere in $t$, and \eqref{eq:gradient-gap} implies that for almost every $t$ in a neighbourhood of $0$,
\[
|(\phi^i)_n'(t)-(\phi^i)_0'(t)|\ge\epsilon_0.
\]
Suppose for example that $(\phi^i)_n'(t)\ge(\phi^i)_0'(0)+\epsilon_0$ (the other case is symmetric). Integrating between $0$ and $s$ (with $s<0$), we obtain
\[
\phi_n^i(0)-\phi_n^i(s)\ge-s\bigl((\phi^i)_0'(0)+\epsilon_0\bigr).
\]
Passing to the limit as $n\to\infty$ gives
\[
\phi^i(0)-\phi^i(s)\ge-s\bigl((\phi^i)_0'(0)+\epsilon_0\bigr).
\]
Dividing by $s<0$ and letting $s\to0^-$, we get
\[
\lim_{s\to0^-}\frac{\phi^i(s)-\phi^i(0)}{s}\ge(\phi^i)_0'(0)+\epsilon_0,
\]
which contradicts the definition of the derivative $(\phi^i)_0'(0)$ (the right-hand limit must equal the left-hand limit). Hence the convergence of gradients holds for almost every $z'$ in a neighbourhood of $O'$. Thus there exists a sequence $z^n\in\cN_{\Om_n}\setminus C$ with $z^n\to x$ and $\nabla\phi_n((z^n)')\to\nabla\phi(x')$. The inward normal is given by
\[
\nu_n(z^n)=\frac{(-\nabla\phi_n((z^n)'),1)}{\sqrt{1+|\nabla\phi_n((z^n)')|^2}},
\]
so by continuity, $\nu_n(z^n)\to\nu(x)$.

Finally, since each ray $z^n+t\nu_n(z^n)$ meets $C$, by the closedness of $C$ and the convergence of $z^n$ and $\nu_n(z^n)$, the limit ray $x+t\nu(x)$ also meets $C$. This proves the lemma.
\end{proof}


\end{document}